\documentclass[oneside,english]{amsart}
\usepackage[T1]{fontenc}
\usepackage[latin9]{inputenc}
\usepackage{amsthm}
\usepackage{setspace}
\usepackage{amssymb}
\usepackage{esint}
\numberwithin{equation}{section} 
\numberwithin{figure}{section} 
\theoremstyle{plain}
 \theoremstyle{definition}
 \newtheorem*{defn*}{Definition}
  \theoremstyle{remark}
  \newtheorem*{claim*}{Claim}
\theoremstyle{plain}
\newtheorem{thm}{Theorem}
  \theoremstyle{remark}
  \newtheorem*{rem*}{Remark}
  \theoremstyle{plain}
  \newtheorem{lem}[thm]{Lemma}
  \theoremstyle{remark}
  \newtheorem{rem}[thm]{Remark}
  \theoremstyle{plain}
  \newtheorem{fact}[thm]{Fact}

\usepackage{babel}

\begin{document}
\global\long\def\bbZ{\mathbb{Z}}

\global\long\def\bbN{\mathbb{N}}

\global\long\def\bbX{\mathbb{X}}

\global\long\def\cF{\mathcal{F}}

\global\long\def\cB{\mathcal{B}}

\global\long\def\fT{\mathfrak{\tau}}

\global\long\def\cCT{\mathcal{C}_{\fT-1}}

\global\long\def\cC{\mathcal{C}}

\title{On A type ${\rm III}_{1}$ bernoulli Shift}

\author{Zemer Kosloff (tel aviv)}

\thanks{This research was was supported by THE ISRAEL SCIENCE FOUNDATION
grant No. 114/08.}

\begin{abstract}
We provide a construction of a product measure under which the shift
on $\left\{ 0,1\right\} ^{\mathbb{Z}}$ is a type ${\rm III}_{1}$
transformation. 
\end{abstract}
\maketitle

\subsection{Introduction: }

A conservative non-singular transformation $T$ of the probability
space $\left(X,\cB,P\right)$ for which there exists no $\sigma$-finite
measure which is $P$-equivalent and $T$-invariant is said to be
of type ${\rm III}$ \cite{KW,Kri}. 

The first constructions of type ${\rm III}$ tranformations were given
in \cite{Orn,Arn}. Both constructions are odometers and hence don't
have the $K$-property. 

In \cite{Ham}, Hamachi introduced a class of product measures $P={\displaystyle \prod_{k=-\infty}^{\infty}}P_{k}$
under which the full shift is a type ${\rm III}$ transformation.
Thus obtaining the first examples of type ${\rm III}$ transformations
which satisfy the $K$-property. 

The classification of the type ${\rm III}$ dynamical systems was
refined by Krieger \cite{Kri} and Araki-Woods who introduced the
ratio set {[}see below{]}. Thus the type ${\rm III}$ transformations
can be further subdivided into type ${\rm III}_{\lambda}$, with $0\leq\lambda\leq1$.
Krieger showed \cite{Kri,KW} that for each $0<\lambda\leq1$ there
is a unique orbit equivalence class. 

In this work we construct a product measure $P={\displaystyle \prod_{k=-\infty}^{\infty}}P_{k}$
under which the full shift is of type ${\rm III}_{1}$. 

\textit{Acknowledgments.} This work is a part of the authors master
thesis done under the supervision of Prof. Jon Aaronson. The author
would like to thank Prof. Jon Aaronson for his time, patience and
helpful suggestions.

\subsection{Preliminaries. }

Let $\left(X,\cB,P\right)$ be a probability space. A measurable transformation
$T:X\to X$ is said to be \textit{non-singular} with respect to $P$
if it preserves the $P-$null sets. i.e for every $A\in\cB,$ if $P(A)=0$
then \[
P\circ T(A):=P\left(TA\right)=0.\]
A measure $m$ on $X$ is said to be $T-$\textit{invariant} if for
every\textit{ $A\in\cB$, \[
P\circ T\left(A\right)=P(A).\]
}The class of non-singular transformations correspond to \textit{automorphisms}
of the measure space\textit{ $(X,\cB,P)$. }

If $T$ is non-singular, then for every $n$, $P\circ T^{n}$ is absolutely
continuous with respect to $P$. By the Radon-Nikodym theorem there
exist measurable functions $\frac{dP\circ T^{n}}{dP}\in L_{1}(X,P)_{+}$
such that for every $A\in\mathcal{B}$ , \[
P\circ T^{n}(A)=\int_{A}\left(\frac{dP\circ T^{n}}{dP}\right)\, dP.\]
For convenience, we will use the notation $\left(T^{n}\right)'(x):=\frac{dP\circ T^{n}}{dP}(x)$
. 

\textbf{Notation} : Let $\left[f\doteq a\pm\epsilon\right]$ stand
for $\left\{ x:\ \left|f(x)-a\right|\leq\epsilon\right\} $. Similarly
\\
$x\doteq e^{\pm a}$ will stand for $e^{-a}<x<e^{a}$.

For every set $A\in\cB$, $\overline{A}$ will denote the compliment
of $A$.

Unless otherwise stated all segments $[a,b]$ will stand for $[a,b]\cap\bbZ$. 

If $\mu$ is a measure on $(X,\cB)$, denote by $\cB_{+}^{\mu}$ the
collection of sets with positive $\mu-$measure. 
\begin{defn*}
\textbf{The Ratio set} $R(T,\mu)$: Let $\left(X,\mathcal{B},\mu,T\right)$
be a non-singular dynamical system. Let $a\geq0$. We shall say that
$a\in R(T,\mu)$ if \[
\forall A\in\cB_{+}^{\mu},\forall\epsilon>0,\exists n\in\bbN,\ A\cap T^{-n}A\cap\left[\left(T^{n}\right)'\doteq a\pm\epsilon\right]\in\cB_{+}^{\mu}.\]

\end{defn*}
The set $R(T,\mu)\cap\mathbb{R}_{+}$ is always a closed multiplicative
subgroup of the positive real numbers. It follows from Maharam's theorem
\cite{Mah} that if $T$ is conservative then $1\in R(T,\mu)$. 

One can show that a system is of type ${\rm III}$ if and only if
$0\in R(T,\mu)$. 

Hence for a type ${\rm III}$ conservative system $R(T,\mu)$ is either
$\{0,1\}$ , \\
$\left\{ \lambda^{n}:n\in\mathbb{Z}\right\} \cup\{0\}$ , for some
$0<\lambda<1$ , or $[0,\infty)$ . These systems are denoted by ${\rm III}_{0}$,
${\rm III}_{\lambda}$ and ${\rm III}_{1}$ respectively.

\medskip{}

\subsection{The shift on the product space:}

Let $\bbX=\{0,1\}^{\bbZ}$ and $T$ be the left shift action on $\bbX$,
that is \[
\left(Tw\right)_{i}=w_{i+1}.\]

Denote the cylinder sets by \[
\left[b\right]_{k}^{l}=\left\{ w\in\bbX:\ \forall i=k,...,l,\ w_{i}=b_{i}\right\} .\]
Let $\cF(k,l)$ denote the smallest $\sigma-$algebra which contains
the cylinders $\left\{ \left[b\right]_{k}^{l}:\ b\in{\displaystyle \prod_{j=l}^{k}}\{0,1\}\right\} $.
For convenience write $\cF(n)$ for $\cF(-n,n)$. 

A measure $P={\displaystyle \prod_{k=-\infty}^{\infty}}P_{k}\in\mathcal{P}(\bbX)$
is called a product measure if for every $k<l$, and for every cylinder
$[b]_{k}^{l},$

\[
P\left(\left[b\right]_{k}^{l}\right)=\prod_{j=k}^{l}P_{j}\left(\left\{ b_{j}\right\} \right).\]

The following claim is a direct consequence of Kakutani's Theorem
on equivalence of product measures\cite{Kak}.
\begin{claim*}
Let $P={\displaystyle \prod_{k=-\infty}^{\infty}P_{k}}$ be a product
measure for which there exists $0<p<1$ such that for every $k\in\bbZ$,
$p<P_{k}\left(\left\{ 1\right\} \right)<1-p.$ Then the shift is non-singular
if and only if\[
\sum_{k=-\infty}^{\infty}\left(P_{k}\left(\left\{ 0\right\} \right)-P_{k-1}\left(\left\{ 0\right\} \right)\right)^{2}<\infty.\]
In addition for every $w\in\bbX$, \[
\left(T^{n}\right)'(w)=\prod_{k=-\infty}^{\infty}\frac{P_{k-n}\left(w_{k}\right)}{P_{k}\left(w_{k}\right)}.\]

\end{claim*}
\bigskip{}

\section{A Type ${\rm III}_{1}$ shift\label{Sec:A-Type-III 1}}

In this section we construct the product measure. 

The product measure will be $P={\displaystyle \prod_{k=-\infty}^{\infty}}P_{k}$
, where \begin{equation}
\forall i\geq0,\ P_{i}\left(0\right)=P_{i}\left(1\right)=\frac{1}{2}.\label{P on N+}\end{equation}

The definition of $P_{k}-$ for negative $k$'s is more complicated
as it involves an inductive procedure.

\subsection{The inductive definition of $P_{k}$ for $k<0$:}

We will need to define inductively 5 sequences $\left\{ \lambda_{t}\right\} _{t=1}^{\infty},\left\{ n_{t}\right\} _{t=1}^{\infty}$
, $\left\{ m_{t}\right\} _{t=1}^{\infty},\ \left\{ M_{t}\right\} _{t=0}^{\infty}$
and $\left\{ N_{t}\right\} _{t=1}^{\infty}$ . The sequence $\left\{ \lambda_{t}\right\} $
is of real numbers which decreases to $1$. The other four, $\left\{ n_{t}\right\} _{t=1}^{\infty}$
, $\left\{ m_{t}\right\} _{t=1}^{\infty},\ \left\{ M_{t}\right\} _{t=0}^{\infty}$
and $\left\{ N_{t}\right\} _{t=1}^{\infty}$ are increasing sequences
of integers. \medskip{}

First choose a positive summable sequence $\left\{ \epsilon_{t}\right\} _{t=1}^{\infty}$
and set $M_{0}=1$. 

\underbar{Base of the induction}: Set $\lambda_{1}=e$ , $n_{1}=2$
, $m_{1}=4$ . Set also $N_{1}=M_{0}+n_{1}=3$ and $M_{1}=N_{1}+m_{1}=7.$
$ $

Given $\left\{ \lambda_{u},n_{u},N_{u},m_{u},M_{u}\right\} _{u=1}^{t-1}$
, we will choose the next level $\left\{ \lambda_{t},n_{t},N_{t},m_{t},M_{t}\right\} $
in the following order. First we choose $\lambda_{t}$ depending on
$M_{t-1}$ and $\epsilon_{t}$. Given $\lambda_{t}$ we will choose
$n_{t}$ and then $N_{t}$ will be defined by \[
N_{t}:=M_{t-1}+n_{t}.\]
Then given $N_{t}$ we will choose $m_{t}$ and finally set \[
M_{t}:=N_{t}+m_{t}.\]
$ $\underbar{Choice of $\lambda_{t}$}: Set $k_{t}:=\left\lfloor \log_{2}\left(\frac{M_{t-1}}{\epsilon_{t}}\right)\right\rfloor +1$
, where $\left\lfloor x\right\rfloor $ denotes the integral part
of $x$. Then set $\lambda_{t}=e^{\frac{1}{2^{k_{t}}}}$. With this
choice of $\lambda_{t}$ we have,\begin{equation}
\lambda_{t}^{M_{t-1}}<e^{\epsilon_{t}}.\label{Constraint on lambda(t)}\end{equation}
 This choice of $\lambda_{t}$ has the property that for every $u<t$,
$\lambda_{u}=\lambda_{t}^{2^{k_{t}-k_{u}}}.$

Define \[
A_{t-1}:=\left\{ \prod_{u=1}^{t-1}\lambda_{u}^{x_{u}}:\ x_{u}\in\left[-n_{u},n_{u}\right]\right\} .\]

The set $A_{t-1}$ has the following properties:
\begin{enumerate}
\item $a\in A_{t-1}$ if and only if $\frac{1}{a}\in A_{t-1}$. 
\item By choosing $x_{t}=0$, one can see that for every $t\in\mathbb{N}$,
$A_{t-1}\subset A_{t}$. 
\item Since for every $u<t$, $\lambda_{t}=\lambda_{u}^{2^{k_{t}-k_{u}}}$,
then $A_{t-1}$ is a subset of \\
$\lambda_{t}^{\bbZ}=\left\{ \lambda_{t}^{l}:\ l\in\mathbb{Z}\right\} $
. In addition $A_{t-1}$ is a finite set. 
\end{enumerate}
\underbar{Choice of $n_{t}$}: Given $\{\lambda_{u},n_{u},m_{u}\}_{u=1}^{t-1}$
and $\lambda_{t}$, the set $A_{t-1}$ is a finite subset of $\lambda_{t}^{\bbZ}$.
Choose $n_{t}$ large enough so it satisfies the following property: 

\textbf{Property {*}}: $n_{t}$ is large enough so that for every
$a,b\in A_{t-1}$, there exists $p=p(a,b)\in\left(\frac{-n_{t}}{4},\frac{n_{t}}{4}\right)$
for which $\lambda_{t}^{p}=b\cdot a$. 

This is possible since $a,b\in A_{t-1}\subset\lambda_{t}^{\bbZ}$. 

Notice that this property is equivalent to \[
\lambda_{t}^{n_{t}/4}\geq\max\left\{ a^{2}:a\in A_{t-1}\right\} .\]
\underbar{Choice of $m_{t}$}: Now that $n_{t}$ is chosen we set
$N_{t}=M_{t-1}+n_{t}$ . Set \begin{equation}
N_{t}\left(2+2^{3N_{t}}\right)=m_{t}.\label{constraint of m(t)}\end{equation}

\subsection{Definition of the product measure}

Let $P={\displaystyle \prod_{k=-\infty}^{\infty}}P_{k}$ , where for
$k\geq0$, \[
P_{k}\left(\left\{ 0\right\} \right)=P_{k}\left(\left\{ 1\right\} \right)=\frac{1}{2}.\]
And for $k<0$ , \begin{equation}
P_{k}(\{0\})=1-P_{k}(\{1\})=\begin{cases}
\frac{1}{1+\lambda_{t}} & if\ -N_{t}<k\leq M_{t-1}\\
\frac{1}{2} & {\rm for\ all}\ -M_{t}<k\leq-N_{t}\end{cases}\label{P on N-}\end{equation}
Define a function $t(\cdot):\mathbb{N}\to\bbN$ by $t(n)=\min\left\{ t\in\bbN:n<N_{t}\right\} $
.

\subsubsection{Statement of the main theorem: }
\begin{thm}
The shift $T:\left(\bbX,\cB(\bbX),P\right)\circlearrowleft$ is non-singular
and of type $III_{1}$. \label{thm: main theorem}\end{thm}
\begin{rem*}
By Kakutani's theorem using the fact that for every\\
 $k\in\bbZ,$$\frac{1}{1+e}<P_{k}\left(\{1\}\right)<\frac{e}{1+e}$
, the non-singularity of the shift under $P$ is equivalent to\[
\sum_{k=-\infty}^{\infty}\left(P_{k}\left(\left\{ 0\right\} \right)-P_{k-1}\left(\left\{ 0\right\} \right)\right)^{2}<\infty.\]
Furthermore it follows from the structure of $P$ that\[
P_{k}\left(\left\{ 0\right\} \right)-P_{k-1}\left(\left\{ 0\right\} \right)=\begin{cases}
\frac{1}{2}-\frac{1}{1+\lambda_{t}} & {\rm if}\ k=-M_{t-1}+1\\
\frac{1}{1+\lambda_{t}}-\frac{1}{2} & {\rm if}\ k=-N_{t}+1\\
0 & {\rm otherwise}\end{cases}.\]
Therefore,\[
\sum_{k=-\infty}^{\infty}\left(P_{k}\left(\left\{ 0\right\} \right)-P_{k-1}\left(\left\{ 0\right\} \right)\right)^{2}=2\sum_{t=1}^{\infty}\left(\frac{\lambda_{t}-1}{2(1+\lambda_{t})}\right)^{2}.\]
This sum converges or diverges together with $\sum_{t=1}^{\infty}\left(\log\lambda_{t}\right)^{2}$
which by (\ref{Constraint on lambda(t)}) is less then $\sum\frac{\epsilon_{t}}{M_{t-1}}<\infty$.
Therefore the shift is non-singular with respect to $P$ and we can
calculate the Radon-Nikodym derivatives of $T$.
\end{rem*}

\subsection{Radon-Nikodym derivatives of the shift: }
\begin{lem}
If $P$ is a product measure which is built by the inductive construction
then\label{lem:first derivative lemma}\[
\lim_{n\to\infty}\left[\left(T^{n}\right)'(w)\left/\left(\prod_{k=-N_{t(n)}+1}^{n-1}\frac{P_{k-n}\left(w_{k}\right)}{P_{k}\left(w_{k}\right)}\right)\right.\right]=1\ {\rm uniformly\ in}\ w\in\bbX.\]
\end{lem}
\begin{proof}
By Kakutani's theorem, for every $w\in\bbX$, \[
\left(T^{n}\right)'(w)=\prod_{k=-\infty}^{\infty}\frac{P_{k-n}\left(w_{k}\right)}{P_{k}\left(w_{k}\right)}=\left(\prod_{k=-\infty}^{-N_{t(n)}}\times\prod_{k=-N_{t(n)}+1}^{n-1}\times\prod_{k=n}^{\infty}\right)\frac{P_{k-n}\left(w_{k}\right)}{P_{k}\left(w_{k}\right)}.\]
 For every $k>n$ and for every $w_{k}\in\{0,1\}$, $P_{k-n}\left(w_{k}\right)=P_{k}\left(w_{k}\right)=\frac{1}{2}$.
Therefore \[
\prod_{k=n}^{\infty}\frac{P_{k-n}\left(w_{k}\right)}{P_{k}\left(w_{k}\right)}=1.\]
Also \begin{equation}
\prod_{k=-\infty}^{-N_{t(n)}}\frac{P_{k-n}\left(w_{k}\right)}{P_{k}\left(w_{k}\right)}=\prod_{u=t(n)+1}^{\infty}\left(\prod_{k=-N_{u}+1}^{-M_{u-1}}\frac{P_{k-n}\left(w_{k}\right)}{P_{k}\left(w_{k}\right)}\times\prod_{k=-M_{u}+1}^{-N_{u}}\frac{P_{k-n}\left(w_{k}\right)}{P_{k}\left(w_{k}\right)}\right).\label{eq: Contribution from the negative tail}\end{equation}
 For every $u>t(n)$, the length of the segment $\left(-N_{u},-M_{u-1}\right]$
is $n_{u}$ and the length of the segment $\left(-M_{u},-N_{u}\right]$
is $m_{u}$. Both $n_{u}$ and $m_{u}$ are greater than $N_{t(n)}$
, hence greater than $n$. 

Therefore since for every $k\in\left(-N_{u}+n,-M_{u-1}\right]$, $P_{k-n}=P_{k}=\left(\frac{1}{1+\lambda_{u}},\frac{\lambda_{u}}{1+\lambda_{u}}\right)$
then \[
\prod_{k=-N_{u}+1}^{-M_{u-1}}\frac{P_{k-n}\left(w_{k}\right)}{P_{k}\left(w_{k}\right)}=\prod_{k=-N_{u}+1}^{-N_{u}+n}\frac{P_{k-n}\left(w_{k}\right)}{P_{k}\left(w_{k}\right)}.\]
Since $n<m_{u}$ then for every $k\in\left(-N_{u},N_{u+n}\right]$,
\[
-M_{u}=-N_{u}-m_{u}<k-n\leq-N_{u}.\]
So $P_{k-n}=\left(\frac{1}{2},\frac{1}{2}\right)$ and $P_{k}=\left(\frac{1}{1+\lambda_{u}},\frac{\lambda_{u}}{1+\lambda_{u}}\right)$.
In conclusion, for every $u>t(n)$, \begin{equation}
\prod_{k=-N_{u}+1}^{-M_{u-1}}\frac{P_{k-n}\left(w_{k}\right)}{P_{k}\left(w_{k}\right)}=\prod_{k=-N_{u}+1}^{-N_{u}+n}\frac{P_{k-n}\left(w_{k}\right)}{P_{k}\left(w_{k}\right)}\label{eq: the contribution from -N_u -M_(u-1)}\end{equation}
\[
=\prod_{k=-N_{u}+1}^{-N_{u}+n}\frac{1/2}{\lambda_{u}^{w_{k}}/(1+\lambda_{u})}=\frac{\lambda_{u}^{-\sum_{k=-N_{u}+1}^{N_{u}+n}w_{k}}}{\left(2\left(1+\lambda_{u}\right)\right)^{n}}.\]
Similarly for every $u>t(n)$,\begin{equation}
\prod_{k=-M_{u}+1}^{-N_{u}}\frac{P_{k-n}\left(w_{k}\right)}{P_{k}\left(w_{k}\right)}=\lambda_{u}^{\sum_{k=-M_{u}+1}^{-M_{u}+n}w_{k}}\left(2\left(1+\lambda_{u}\right)\right)^{n}.\label{eq:Contribution from -M_u -N_u}\end{equation}
Putting together \eqref{eq: Contribution from the negative tail},
\eqref{eq: the contribution from -N_u -M_(u-1)} and \eqref{eq:Contribution from -M_u -N_u}
we have\[
\prod_{k=-\infty}^{-N_{t(n)}}\frac{P_{k-n}\left(w_{k}\right)}{P_{k}\left(w_{k}\right)}=\]
\[
=\prod_{u=t(n)+1}^{\infty}\lambda_{u}^{\sum_{k=-M_{u-1+1}}^{-M_{u-1}+n}w_{k}-\sum_{k=-N_{u}+1}^{-N_{u}+n}w_{k}}\doteq\prod_{u=t(n)+1}^{\infty}\lambda_{u}^{\pm n}.\]
Since for every $u>t(n).\ n<M_{u-1}$ , it follows that \[
\lambda_{u}^{\pm n}\doteq\lambda_{u}^{\pm M_{u-1}}\doteq e^{\pm\epsilon_{u}}.\]
So $ $ $\prod_{k=-\infty}^{-N_{t(n)}}\frac{P_{k-n}\left(w_{k}\right)}{P_{k}\left(w_{k}\right)}\doteq e^{\pm\sum_{u=t(n)+1}^{\infty}\epsilon_{u}}\to1$
as $n\to\infty$, since $\sum\epsilon_{t}<\infty$. 

The conclusion follows. \end{proof}
\begin{defn*}
For every $u\in\bbN$ , write\[
\Upsilon_{u}(w)=\sum_{k=-N_{u}+1}^{-M_{u-1}}w_{k}=\#\left\{ k\in\left(-N_{u},-M_{u-1}\right]:\ w_{k}=1\right\} .\]

\end{defn*}
Set also \[
f_{t}(w)=\prod_{u=1}^{t}\lambda_{u}^{\Upsilon_{u}(w)}.\]
$f_{t}(\cdot)$ is $\cF\left(-N_{t},0\right)$ measurable. 

The next lemma uses the previous lemma to show that if $n$ belongs
to an interval of the form $\left[N_{t-1},M_{t-1}\right)$ then we
can obtain a reasonable approximation of $\left(T^{n}\right)'$ .
This approximation becomes more accurate as $t\to\infty$. Notice
that for $n\in\left[N_{t-1},M_{t-1}\right)$, $t=t(n)$.
\begin{lem}
For every $N_{t-1}\leq n<M_{t-1}$ and for every $w\in\bbX$,\label{lem:second derivative lemma}
\[
\left(T^{n}\right)'(w)=K_{n}(w)\cdot\frac{f_{t(n)-1}\circ T^{n}\left(w\right)}{f_{t(n)-1}(w)},\]
Where $\lim_{n\to\infty}K_{n}(w)=1$ uniformly in $w\in\bbX$. \end{lem}
\begin{proof}
By lemma \ref{lem:first derivative lemma} it is enough to show that
for $N_{t-1}\leq n<m_{t-1}$, \[
\prod_{k=-N_{t(n)}+1}^{n-1}\frac{P_{k-n}\left(w_{k}\right)}{P_{k}\left(w_{k}\right)}=h_{n}(w)\cdot\prod_{u=1}^{t(n)-1}\lambda_{u}^{\Upsilon_{u}\circ T^{n}\left(w\right)-\Upsilon_{u}(w)}.\]
Where $\lim_{n\to\infty}h_{n}(w)=1$ uniformly in $w\in\bbX$ . 

First notice that because of the product structure \[
\prod_{k=-N_{t(n)}+1}^{n-1}P_{k}\left(w_{k}\right)=P\left(\left[w\right]_{-N_{t(n)}+1}^{n-1}\right).\]
And \[
P\left(\left[w\right]_{-N_{t(n)}}^{n-1}\right)=P\left(\left[w\right]_{0}^{n-1}\right)\cdot\prod_{u=1}^{t(n)}\left(P\left(\left[w\right]_{-M_{u}+1}^{-N_{u}}\right)\cdot P\left(\left[w\right]_{-N_{u}+1}^{-M_{u-1}}\right)\right).\]
By \eqref{P on N+} , $ $$P\left(\left[w\right]_{0}^{n-1}\right)=\frac{1}{2^{n}}$.
By \eqref{P on N-} , \[
P\left(\left[w\right]_{-M_{u}+1}^{-N_{u}}\right)\cdot P\left(\left[w\right]_{-N_{u}+1}^{-M_{u-1}}\right)=\frac{1}{2^{m_{u}}}\cdot\left(\frac{1}{1+\lambda_{u}}\right)^{n_{u}}\lambda_{u}^{\sum_{u=-N_{u}+1}^{-M_{u-1}}w_{k}}=\]
\[
=\frac{1}{2^{m_{u}}}\cdot\left(\frac{1}{1+\lambda_{u}}\right)^{n_{u}}\lambda_{u}^{\Upsilon_{u}(w)}.\]
Multiplying all the terms\begin{equation}
\prod_{k=-N_{t(n)}+1}^{n-1}P_{k}\left(w_{k}\right)=2^{-n+\sum_{u=1}^{t(n)}m_{u}}\cdot\prod_{u=1}^{t(n)}\left(\frac{1}{1+\lambda_{u}}\right)^{n_{u}}\lambda_{u}^{\Upsilon_{u}(w)}.\label{estimate on prod P_k}\end{equation}
Similarly \[
\prod_{k=-N_{t(n)}+1}^{n-1}P_{k-n}\left(w_{k}\right)=\prod_{k=-N_{t(n)}-n+1}^{0}P_{k}\left(w_{k+n}\right)\]
\[
=\prod_{k=-N_{t(n)}-n}^{-N_{t(n)}}P_{k}\left(w_{k+n}\right)\cdot\prod_{u=1}^{t(n)}\left[\prod_{k=-M_{u}+1}^{-N_{u}}\times\prod_{k=-N_{u}+1}^{-M_{u-1}}P_{k}\left(w_{k+n}\right)\right].\]
By the definition of $P_{k}$ for negative $k$,\[
\left\{ \begin{array}{lll}
\prod_{k=-N_{t(n)}-n}^{-N_{t(n)}}P_{k}\left(w_{k+n}\right) & = & \prod_{k=-N_{t(n)}-n}^{-N_{t(n)}}\frac{1}{2}=\frac{1}{2^{n}}\\
\prod_{k=-M_{u}+1}^{-N_{u}}P_{k}\left(w_{k+n}\right) & = & \frac{1}{2^{m_{u}}}\\
\prod_{k=-N_{u}+1}^{-M_{u-1}}P_{k}\left(w_{k+n}\right) & = & \prod_{k=-N_{u}+1}^{-M_{u-1}}\frac{\lambda_{u}^{w_{k+n}}}{1+\lambda_{u}}=\frac{\lambda_{u}^{\sum_{k=-N_{u}+1}^{-M_{u-1}}w_{k+n}}}{(1+\lambda_{u})^{n_{u}}}\end{array},\right.\]
therefore\[
\prod_{k=-N_{t}+1}^{n-1}P_{k-n}\left(w_{k}\right)=2^{-n+\sum_{u=1}^{t(n)}m_{u}}\cdot\prod_{u=1}^{t(n)}\left(\frac{1}{1+\lambda_{u}}\right)^{n_{u}}\lambda_{u}^{\sum_{k=-N_{u}+1}^{-M_{u-1}}w_{k+n}}\]
\begin{equation}
=2^{-n+\sum_{u=1}^{t(n)}m_{u}}\cdot\prod_{u=1}^{t(n)}\left(\frac{1}{1+\lambda_{u}}\right)^{n_{u}}\lambda_{u}^{\Upsilon_{u}\circ T^{n}\left(w\right)}.\label{estimate on prod P_(k-n)}\end{equation}
 Combining the estimates \eqref{estimate on prod P_k} and \eqref{estimate on prod P_(k-n)}
, we get , \[
\prod_{k=-N_{t(n)}+1}^{n-1}\frac{P_{k-n}\left(w_{k}\right)}{P_{k}\left(w_{k}\right)}=\prod_{u=1}^{t(n)}\lambda_{u}^{\Upsilon_{u}\circ T^{n}(w)-\Upsilon_{u}(w)}=\lambda_{t(n)}^{\Upsilon_{t(n)}\circ T^{n}(w)-\Upsilon_{t(n)}(w)}\cdot\frac{f_{t(n)}\circ T^{n}\left(w\right)}{f_{t(n)}(w)}.\]
The conclusion will follow once we show that for every

$\begin{gathered}N_{t-1}\leq n<M_{t-1}\ ,\ h_{n}(w):=\lambda_{t(n)}^{\Upsilon_{t(n)}\circ T^{n}(w)-\Upsilon_{t(n)}(w)}\to1\end{gathered}
$ uniformly in $w\in\bbX$ as $n\to\infty$. 

Notice that, \[
\Upsilon_{t(n)}(T^{n}w)-\Upsilon_{t(n)}(w)=\sum_{u=-N_{t(n)}-1}^{-M_{t(n)-1}}\left(w_{k+n}-w_{k}\right)=\]
 \[
\sum_{k=-M_{t(n)-1}+n}^{-M_{t(n)-1}+1}w_{k}-\sum_{k=-N_{t(n)}-1}^{-N_{t(n)}+n}w_{k}\doteq\pm n\doteq\pm M_{t(n)-1}.\]
Therefore\[
\lambda_{t(n)}^{\Upsilon_{u}\circ T^{n}(w)-\Upsilon_{u}(w)}\doteq\lambda_{t(n)}^{\pm M_{t(n)-1}}\doteq e^{\pm\epsilon_{t(n)}}\to1.\]
 The conclusion follows.
\end{proof}
The following two remarks will be useful in the proof of type ${\rm III}_{1}$. 
\begin{rem*}
Since ${\rm Image}\left(\Upsilon_{u}\circ T^{n}-\Upsilon_{u}\right)=\left[-n_{u},n_{u}\right]$,
the set $A_{t}$ is the set of all values of \[
\frac{f_{t}\circ T^{n}(w)}{f_{t}(w)}=\prod_{u=1}^{t}\lambda_{u}^{\Upsilon_{u}\circ T^{n}(w)-\Upsilon_{u}(w)}.\]

\end{rem*}
Since for every $t_{0}<t$, $\lambda_{t_{0}}\in A_{t-1}$ and also
$\frac{1}{f_{t-1}(w)}\in A_{t}$ we can formulate (property {*}) in
the following way. 
\begin{rem}
\textbf{\label{rem:A-reformulation-of property *}}It follows from
property ({*}) that $n_{t}$ is large enough so that for every $w\in\bbX$,
$t_{0}<t$ and $k\in\bbN$ there exists $p=p(w)\in\left(-\frac{n_{t}}{4},\frac{n_{t}}{4}\right)$
so that \[
\lambda_{t}^{p}=\frac{\lambda_{t_{0}}}{f_{t-1}(w)}.\]

\end{rem}

\subsection{The proof of type ${\rm III}_{1}$}

Since $\lambda_{t}\downarrow1$ , and $\mathcal{R}(T)$ is always
a closed multiplicative subgroup of $\mathbb{R}_{+}$, the following
lemma will yield that $\left(\bbX,\cB,P,T\right)$ is of type $III_{1}$. 
\begin{lem}
For every $t\in\bbN$, $\lambda_{t}$ belongs to the ratio set of
$T$.\label{lem:lambda(t) are essential values}
\end{lem}
The proof of lemma \ref{lem:lambda(t) are essential values} will
be a series of lemmas. We will show that $\lambda_{t}$ satisfies
an $EVC$ property as in \cite{A-L--1} , where the generating partitions
of $\mathbb{X}$, will be cylinders $\{\left[b\right]_{-n}^{n}:b\in\{-1,1\}^{n}\}$
and the neighborhoods of $\lambda_{t}$ will be $U_{\epsilon}=\left[\lambda_{t}-\epsilon,\lambda_{t}+\epsilon\right]$. 
\begin{lem}
For every $t_{0}\in\bbN$ the following holds: \label{lem: EVC property for the shift}

\begin{onehalfspace}
For every cylinder set $B=[b]_{-n}^{n}$ and for every $\epsilon>0$,
there exists a $\tau=\fT(B,t_{0},\epsilon)\in\bbN$, for which\begin{equation}
P\left(\cup_{l=1}^{\frac{m_{\fT}}{N_{\fT}}}\left\{ B\cap T^{-lN_{\fT}}B\cap\left[\left(T^{lN_{\fT}}\right)^{'}\doteq\lambda\pm\epsilon\right]\right\} \right)\geq0.9\cdot P(B).\label{EVC}\end{equation}
\end{onehalfspace}

\end{lem}
\begin{onehalfspace}
Throughout the  proofs and the lemmas the letter $n$ will denote
the radius of the cylinder. 

Since the proof is delicate and involves many details we give a sketch
of the idea of the proof.

\textbf{Sketch of the proof of Lemma \ref{lem: EVC property for the shift}}:
Let $B=[b]_{-n}^{n}$ , $t_{0}\in\bbN$ and $\epsilon>0$. 

We will first choose $\fT$ large enough so that for every $l=2,\ldots,\frac{m_{\fT}}{N_{\fT}}$,
\[
\left\{ f_{t}\circ T^{lN_{\fT}}(w)=f_{\fT}(w)\cdot\lambda_{t_{0}}\right\} \subset\left[\left(T^{lN_{\fT}}\right)^{'}\doteq\lambda\pm\epsilon\right].\]
Since $P$ is a product measure and $f_{\fT}$ is $\mathcal{F}\left(-N_{\fT},0\right)$
measurable, if we restrict our attention to cylinders of the form
$C=[c]_{-N_{\fT}}^{0}$then \[
C\cap\left\{ f_{t}\circ T^{lN_{\fT}}(w)=f_{\fT}(w)\cdot\lambda_{t_{0}}\right\} =C\cap\left\{ f_{t}\circ T^{lN_{\fT}}(w)=f_{\fT}(c)\cdot\lambda_{t_{0}}\right\} .\]
So we will decompose $B$ to $\left\{ B\cap[c]_{-N_{\fT}}^{0}:c\in{\displaystyle \prod_{k=-N_{\fT-1}}^{0}\{0,1\}}\right\} $
and prove that for the cylinders $C=[c]_{-N_{\fT}}^{0}$ which cover
most of $\bbX$, \[
P\left(\left\{ f_{t}\circ T^{lN_{\fT}}(w)=f_{\fT}(c)\cdot\lambda_{t_{0}}\right\} \right)\geq\frac{1}{2^{N_{\fT}}}.\]
Then we make use of the independence of $B\cap C,\left\{ T^{-lN_{\fT}}B\right\} _{l=2}^{\frac{m_{\fT}}{N_{\fT}}}$
and $\left\{ \left[f_{\fT}\circ T^{lN_{\fT}}=f_{\fT}(c)\right]\right\} _{l=2}^{\frac{m_{\fT}}{N_{\fT}}}$
to get \[
P\left(\cup_{l=1}^{\frac{m_{\fT}}{N_{\fT}}}\left\{ \left(B\cap{\bf C}\right)\cap T^{-lN_{\fT}}B\cap\left[\left(T^{lN_{\fT}}\right)^{'}\doteq\lambda\pm\epsilon\right]\right\} \right)\geq(1-\epsilon)P(B\cap C).\]
Afterwards we will finish the proof by summing over $C$.

\bigskip{}

\end{onehalfspace}
\begin{fact}
Since for every $u\in\bbN$, $1<\lambda_{u}\leq e$, we have that
for every $u\in\bbN$, \[
\frac{1}{2}<\frac{\lambda_{u}}{1+\lambda_{u}}<\frac{e}{1+e}.\]
and since $\Upsilon_{u}(w)\sim{\rm Bin}\left(n_{u},\frac{\lambda_{u}}{1+\lambda_{u}}\right)$,
the strong law of large numbers implies that \[
\lim_{u\to\infty}P\left(\left\{ w:\Upsilon_{u}(w)\in\left[\frac{n_{u}}{4},\frac{3n_{u}}{4}\right]\right\} \right)=1.\]
\medskip{}

Write $\cC_{t}$ for the collection of cylinders of the form $C=\left[c\right]_{-N_{t}}^{-1}$
such that\[
\Upsilon_{t}(c)=\sum_{k=-N_{t}+1}^{-M_{t-1}}c_{k}\in\left[\frac{n_{t}}{4},\frac{3n_{t}}{4}\right].\]
It follows that \begin{equation}
\lim_{t\to\infty}P\left(\cup_{C\in\cC_{t}}C\right)=1.\label{C(t) are dense cylinders by SLLN}\end{equation}
Furthermore since $f_{t}$ is $\cF\left(-N_{t},0\right)$ measurable
then for all cylinders $C=[c]_{-N_{t}}^{0}\in\cC_{t}$, \[
\left.f_{t}\right|_{C}=f_{t}(c).\]
\end{fact}
\begin{lem}
\label{Lem: inequlity which helps with the derivative}Let $t_{0}\in\bbN$.
For every $t$ such that $t>t_{0}+1$ the following holds:

For every cylinder $C=[c]_{-N_{t}}^{0}\in\cC_{t}$ and for every $l=2,\ldots,\frac{m_{t}}{N_{t}}$,
\[
P\left(\left\{ w:f_{t}\circ T^{lN_{\fT}}(w)=\lambda_{t_{0}}f_{t}(c)\right\} \right)\geq\frac{1}{2^{N_{t}}}.\]
 \end{lem}
\begin{proof}
\begin{onehalfspace}
Let $C\in\cC_{t}$ and $2\leq l\leq\frac{m_{t}}{N_{t}}$.

We will build a cylinder set $D=D(l,C,t_{0})=\left[d\right]_{\left(l-1\right)N_{t}+1}^{lN_{t}}$
such that\[
D\subset\left\{ w:f_{t}\circ T^{lN_{t}}(w)=f_{t}(c)\cdot\lambda_{t_{0}}\right\} .\]
Notice that \begin{equation}
\frac{f_{t}\circ T^{lN_{t}}(w)}{f_{t}(c)}=\lambda_{t-1}^{\Upsilon_{t}\circ T^{lN_{t}}(w)-\Upsilon_{t}(c)}\cdot\frac{f_{t-1}\circ T^{lN_{t}}(w)}{f_{t-1}(c)}\label{eq: First form of f_t /f_t}\end{equation}
Set \[
d_{k}=0\ {\rm for\ every\ }k\in\left(\left(l-1\right)N_{t}+M_{t-1}+1,lN_{t}\right].\]
It is easy to see that for every $u<t-1$, $\Upsilon_{u}\circ T^{lN_{t}}$
is $\cF\left(\left(l-1\right)N_{t}+M_{t-1},lN_{t}\right]$ measurable
and therefore for every $w\in D,u<t-1$, \[
\Upsilon_{u}\circ T^{lN_{t}}(w)=\sum_{k=lN_{t}-N_{u}+1}^{lN_{t}-M_{u-1}}d_{k}=0.\]
So for every $w\in D,$ \[
f_{t-1}\circ T^{lN_{t}}(w)=\prod_{u=1}^{t-1}\lambda_{u}^{\Upsilon_{u}\circ T^{lN_{t}(w)}}=1\]
and \[
\frac{f_{t}\circ T^{lN_{t}}(w)}{f_{t}(c)}=\lambda_{t-1}^{\Upsilon_{t}\circ T^{lN_{t}}(w)-\Upsilon_{t}(c)}\cdot\frac{1}{f_{t-1}(c)}.\]
Since $n_{t}$ satisfies property ({*}) {[} see Remark \ref{rem:A-reformulation-of property *}{]},
there exists a $p\in\left(-\frac{n_{t-1}}{4},\frac{n_{t-1}}{4}\right)$
such that\begin{equation}
\lambda_{t}^{p}=\lambda_{t_{0}}\cdot\left(\frac{1}{f_{t-1}(c)}\right).\label{p= prod/prod}\end{equation}
Also since $C\in\cC_{t}$ , $\Upsilon_{t}(c)\in\left(\frac{n_{t}}{4},\frac{3n_{t}}{4}\right)$
and so \[
p\in\left(-\frac{n_{t}}{4},\frac{n_{t}}{4}\right)\subset{\rm {\displaystyle Image}}\left(\Upsilon_{t}\circ T^{lN_{t}}(w)-\Upsilon_{t}(c):w\in\bbX\right).\]
Since $\Upsilon_{t}\circ T^{lN_{t}}(w)$ is $\cF\left(\left(l-1\right)N_{t-1}+1,\left(l-1\right)N_{t}+M_{t-1}-1\right)$
measurable this means that there exists \[
\tilde{w}=\left\{ \tilde{w}_{k}\right\} _{k=\left(l-1\right)N_{t}+1}^{\left(l-1\right)N_{t}+M_{t-1}}\in\prod_{k=lN_{t}-N_{t-1}}^{lN_{t}-M_{t-2}}\{0,1\}\]
such that \[
p=\Upsilon_{t}\circ T^{lN_{t}}({\bf \tilde{w}})-\Upsilon_{t-1}(c).\]
By setting $d_{k}=\tilde{w}_{k}$ for every $ $$lN_{t}-N_{t-1}\leq k\leq lN_{t}-M_{t-2}$
we have finished the construction of $D$. 

Observe that for every $w\in D$,\[
\frac{f_{\fT}\circ T^{lN_{t}}(w)}{f_{t}(c)}=\frac{f_{t}\circ T^{lN_{t}}(d)}{f_{t}(c)}=\lambda_{t}^{\Upsilon_{t}\circ T^{lN_{t}}(\tilde{w})-\Upsilon_{t}(c)}\cdot\frac{1}{f_{t-1}(c)}\]
\[
=\lambda_{t}^{p}\cdot\frac{1}{f_{t-1}(c)}=\lambda_{t_{0}}.\]
Therefore \[
D\subset\left\{ w:f_{t}\circ T^{lN_{t}}(w)=f_{t}(c)\cdot\lambda_{t_{0}}\right\} .\]
Since $P(D)={\displaystyle \prod_{k=lN_{\fT}}^{\left(l+1\right)N_{\fT}}}P_{k}\left(d_{k}\right)=\frac{1}{2^{N_{t}}}$
the lemma follows.\end{onehalfspace}

\end{proof}
Notation: Given a cylinder $C=[c]_{-N_{t-1}}^{0}$ , $t_{0}\in\bbN$
and $l=2,\ldots,\frac{m_{t}}{N_{t}}$ , we will abuse notation and
write $D(C,t_{0},l)$ for $\left\{ w:f_{t}\circ T^{lN_{t}}(w)=f_{t}(c)\cdot\lambda_{t_{0}}\right\} .$
$ $
\begin{lem}
Let $B=[b]_{-n}^{n}$, and $t_{0}\in\bbN$. Then for every $t\in\bbN$
such that $2n<N_{t}$ and $t_{0}<t$ the following holds:

For every $C\in\cC_{t}$ \begin{equation}
P\left((B\cap C)\cap\overline{\left(\cup_{l=2}^{\frac{m_{t}}{N_{t}}}\left[T^{-lN_{t}}B\cap D\left(C,t_{0},l\right)\right]\right)}\right)\leq\left(1-\frac{1}{2^{2N_{t}}}\right)^{2^{3N_{t}}}P(B\cap C).\label{eq: Almost EVC for special cylinders}\end{equation}
\end{lem}
\begin{proof}
\begin{onehalfspace}
Notice that $f_{t}$  is $\cF\left(-N_{t-1},0\right)$ measurable
and therefore for every $l\leq\frac{m_{t}}{N_{t}}$, $f_{t}\circ T^{lN_{t}}$
is $\cF\left(\left(l-1\right)N_{t},lN_{t}\right)$ measurable. By
the product structure of $P$ , the random variables $\left\{ f_{t}\circ T^{lN_{t}}\right\} _{l=0}^{\frac{m_{t}}{N_{t}}}$
are independent. Therefore the sets $\left\{ D\left(C,t_{0},l\right)\right\} _{l=2}^{\frac{m_{t}}{N_{t}}}$
are independent and each one is measurable $\cF(N_{t},\infty).$ 

Since $B\in\mathcal{F}(-n,n)$, then for every $l\geq2$, $T^{-lN_{t}}B\in\cF\left(-n-lN_{t},n-lN_{t}\right)$.
Since $2n<N_{t}$ the sets $\left\{ T^{-lN_{t}}B\right\} _{l=2}^{\frac{m_{t}}{N_{t}}}$
are independent and independent of $\mathcal{F}\left(-N_{t},\infty\right)$.
So the sets $\left\{ D\left(C,t_{0},l\right)\right\} _{l=2}^{\frac{m_{t}}{N_{t}}},(B\cap C)$
and $\left\{ T^{-lN_{t}}B\right\} _{l=2}^{\frac{m_{t}}{N_{t}}}$ are
independent. \\
Therefore for every $l=2,...,\frac{m_{t}}{N_{t}}$\[
P\left(T^{-lN_{t}}B\cap D\left(C,t_{0},l\right)\right)=P\left(T^{-lN_{t}}B\right)P\left(D\left(C,t_{0},l\right)\right).\]
Plugging in Lemma \ref{Lem: inequlity which helps with the derivative}
we have \[
P\left(T^{-lN_{t}}B\cap D\left(C,t_{0},l\right)\right)\geq\frac{P\left(T^{-lN_{t}}B\right)}{2^{N_{t}}}\]
Furthermore, since for every $2\leq l\leq\frac{m_{t}}{N_{t}}-1,$
\[
\left[-n-lN_{t},n-lN_{t}\right]\subset\left(-M_{t},-N_{t}\right],\]
the product structure of $P$ (see \eqref{P on N-}) implies \begin{equation}
P\left(T^{-lN_{t}}B\right)=\frac{1}{2^{2n}}\geq\frac{1}{2^{N_{t}}},\ {\rm for\ every}\ l=2,...,\frac{m_{t}}{N_{t}}-1.\label{eq: T^-N B is large}\end{equation}
Therefore

\begin{equation}
P\left(\overline{\left[T^{-lN_{t}}B\cap D\left(C,t_{0},l\right)\right]}\right)\leq\left(1-\frac{1}{2^{2N_{t}}}\right).\label{eq:the compliment is small}\end{equation}

Since the sets $\left\{ T^{-lN_{t}}B\cap D\left(C,t_{0},l\right)\right\} _{l=2}^{\frac{m_{t}}{N_{t}}},\left(B\cap C\right)$
are independent ( and \eqref{eq:the compliment is small}) ,\[
P\left(B\cap C\cap\overline{\left(\cup_{l=2}^{\frac{m_{t}}{N_{t}}-1}\left[T^{-lN_{t}}B\cap D\left(C,t_{0},l\right)\right]\right)}\right)\]
\[
=P\left(B\cap C\right)P\left(\cap_{l=2}^{\frac{m_{t}}{N_{t}}}\overline{\left[T^{-lN_{t}}B\cap D\left(C,t_{0},l\right)\right]}\right)\]
\[
=P\left(B\cap C\right)\cdot\prod_{l=2}^{\frac{m_{t}}{N_{t}}-1}P\left(\overline{\left[T^{-lN_{t}}B\cap D\left(C,t_{0},l\right)\right]}\right).\]
\[
\leq\left(1-\frac{1}{2^{2N_{t}}}\right)^{\frac{m_{t}}{N_{t}}-2}P(B\cap C)=\left(1-\frac{1}{2^{2N_{t}}}\right)^{2^{3N_{t}}}P(B\cap C).\]
\end{onehalfspace}

\end{proof}
\begin{proof}
\begin{onehalfspace}
\textsf{\textit{(Lemma \ref{lem: EVC property for the shift}}}):
Let $B=[b]_{-n}^{n}$ be fixed, $\epsilon>0$ and $t_{0}\in\bbN$.
We want to choose a $\fT\in\bbN$ such that \eqref{EVC} holds. 

\textsf{\textit{The choice of $\fT$:}} Choose a $\fT\in\bbN$ which
satisfies the following properties:\begin{equation}
\fT>t_{0},\label{How Large is T_1}\end{equation}
(By this condition $\lambda_{t_{0}}\in A_{\fT}$ )\[
2n<N_{\fT},\]
\begin{equation}
P\left(\cup_{C\in\cC_{\fT}}C\right)\geq1-0.05\cdot P(B),\label{How large is T_2}\end{equation}
(so {}``nice'' cylinders exhaust enough of $B$) and \begin{equation}
\left(1-\frac{1}{2^{2N_{\fT}}}\right)^{2^{3N_{\fT}}}<0.01.\label{How large is T_3}\end{equation}

By Lemma \ref{lem:second derivative lemma} we can, by enlarging $\fT$
if needed , demand in addition to (\ref{How Large is T_1}).(\ref{How large is T_2})
and (\ref{How large is T_3}) that $\fT$ satisfies the following
: For every\\
 $N_{\fT}\leq k<M_{\fT}$,\[
\forall w\in\bbX.\left(T^{k}\right)'(w)\doteq\left(1\pm\frac{\epsilon}{3}\right)\frac{f_{\fT}\left(T^{k}w\right)}{f_{\fT}(w)}.\]
This means (derivative approximation),\begin{equation}
\forall1\leq l\leq\frac{m_{\fT}}{N_{\fT}},\ \left(T^{lN_{\fT}}\right)'(\cdot)\doteq\left(1\pm\frac{\epsilon}{3}\right)\frac{f_{\fT}\circ T^{lN_{\fT}}\left(\cdot\right)}{f_{\fT}(\cdot)}.\label{How Large is T_4}\end{equation}
We will now prove that this $\fT$ satisfies (\ref{EVC}) with $B,\epsilon$
and $t_{0}$. 

First we will show that for every $C=[c]_{-N_{\fT}}^{0}\in C_{\fT}$,
\[
P\left(\cup_{l=1}^{\frac{m_{\fT}}{N_{\fT}}}\left\{ \left(B\cap C\right)\cap T^{-lN_{\fT}}B\cap\left[\left(T^{lN_{\fT}}\right)^{'}\doteq\lambda\pm\epsilon\right]\right\} \right)\geq0.99\cdot P(B\cap C).\]
\medskip{}

Since $\fT$ satisfies \eqref{How Large is T_4} and $\lambda_{t_{0}}<3$
then for every $l=2,\ldots,\frac{m_{\fT}}{N_{\fT}}$, \[
\left\{ w:f_{\fT}\circ T^{lN_{\fT}}(w)=f_{\fT}(w)\cdot\lambda_{t_{0}}\right\} \subset\left[\left(T^{lN_{\fT}}\right)^{'}\doteq\lambda\pm\epsilon\right].\]
Since $f_{\fT}$ is $\cF\left(-N_{\fT},0\right)$ measurable, \[
C\cap\left\{ w:f_{\fT}\circ T^{lN_{\fT}}(w)=f_{\fT}(w)\cdot\lambda_{t_{0}}\right\} =C\cap\left\{ w:f_{\fT}\circ T^{lN_{\fT}}(w)=f_{\fT}(c)\cdot\lambda_{t_{0}}\right\} \]
\[
=C\cap D\left(C,t_{0},l\right).\]
Therefore $\cup_{l=1}^{\frac{m_{\fT}}{N_{\fT}}}\left\{ \left(B\cap C\right)\cap T^{-lN_{\fT}}B\cap D\left(C,t_{0},l\right)\right\} $
is a subset of\\
 $\cup_{l=1}^{\frac{m_{\fT}}{N_{\fT}}}\left\{ \left(B\cap C\right)\cap T^{-lN_{\fT}}B\cap\left[\left(T^{lN_{\fT}}\right)^{'}\doteq\lambda\pm\epsilon\right]\right\} $.

So by Lemma \ref{Lem: inequlity which helps with the derivative},\[
P\left((B\cap C)\cap\overline{\left(\cup_{l=1}^{\frac{m_{\fT}}{N_{\fT}}}T^{-lN_{\fT}}B\cap\left[\left(T^{lN_{\fT}}\right)^{'}\doteq\lambda\pm\epsilon\right]\right)}\right)\]
 \[
\leq P\left((B\cap C)\cap\overline{\left(\cup_{l=1}^{\frac{m_{\fT}}{N_{\fT}}}T^{-lN_{\fT}}B\cap D\left(C,t_{0},l\right)\right)}\right)\]
\[
\leq\left(1-\frac{1}{2^{2N_{\fT}}}\right)^{2^{3N_{\tau}}}P(B\cap C)<0.01\cdot P(B\cap C).\]
\[
\]
 The lemma follows by ( the last inequality uses \ref{How large is T_2})
\[
P\left(B\cap\overline{\left(\cup_{l=1}^{\frac{m_{\fT}}{N_{\fT}}}\left[T^{-lN_{\fT}}B\cap\left\{ \left(T^{lN_{\fT}}\right)^{'}(\cdot)=\lambda_{t_{0}}\pm\epsilon\right\} \right]\right)}\right)\leq P\left(\overline{\left[\cup_{C\in\cC_{\fT}}C\right]}\right)+\]
\[
\sum_{C\in\cC_{\fT-1}}P\left(B\cap C\cap\overline{\left(\cup_{l=1}^{\frac{m_{\fT}}{N_{\fT}}}\left[T^{-lN_{\fT}}B\cap\left[\left(T^{lN_{\fT}}\right)^{'}(\cdot)=\lambda_{t_{0}}\pm\epsilon\right]\right]\right)}\right)<0.1\cdot P(B).\]
\end{onehalfspace}

\end{proof}
Then next lemma is a standard lemma in measure theory. The proof is
included here for the sake of completeness, 
\begin{lem}
For all $\epsilon>0$ and for every $A\in\mathcal{B}_{+}$ there exists
a cylinder of the form $B=\left[b_{-n},...,b_{n}\right]_{0}$ for
which\label{lem:Density of Cylinders} \[
\frac{P\left(A\cap B\right)}{P(B)}>1-\epsilon.\]
\end{lem}
\begin{proof}
\begin{onehalfspace}
\begin{flushleft}
Define a partition $\Sigma_{n}=\left\{ Cylinders\ of\ the\ form\ \left[b_{-n},...,b_{n}\right]_{0}\right\} $
. Denote by $F_{n}=\sigma\left(\Sigma_{n}\right)$ the sigma algebra
generated by $\Sigma_{n}$. Then $\mathcal{F}_{n}\uparrow\mathcal{B}$
therefore for every $A\in\mathcal{B}_{+}$ we have \[
E\left(1_{A}\left|\mathcal{F}_{n}\right.\right)\longrightarrow1_{A}\ a.e\ as\ n\to\infty.\]
 This means that if $P(A)>0$ then there exists an $n\in\mathbb{N}$
and an atom of $\Sigma_{n}$ which we will denote by $B_{n}$ for
which \[
E\left(\left.1_{A}\right|B_{n}\right)>1-\epsilon.\]
 But \[
E\left(\left.1_{A}\right|B_{n}\right)=P\left(A\left|B_{n}\right.\right)=\frac{P\left(A\cap B_{n}\right)}{P(B_{n})}.\]
 By setting $B=B_{n}$ the lemma is proved. 
\par\end{flushleft}\end{onehalfspace}

\end{proof}
\begin{proof}
\textit{(Lemma \ref{lem:lambda(t) are essential values}),}

Let $t_{0}\in\bbN$ such that $\lambda_{t_{0}}<1.4$, we will show
that $\lambda_{t_{0}}$ belongs to the ratio set of the shift.

Let $0<\epsilon<0.1$ and $A\in\cB_{+}$ . By lemma \ref{lem:Density of Cylinders}
there exists a cylinder $B=\left[b\right]_{-n}^{n}$ for which \[
P(A\cap B)\geq0.9\cdot P(B)\ .\]
By Lemma \ref{lem: EVC property for the shift} there exists $\fT\in\mathbb{N}$
for which,\[
P\left(\cup_{l=1}^{\frac{m_{\fT}}{N_{\fT}}}\left[B\cap T^{-lN_{\fT}}B\cap\left\{ \left(T^{lN_{\fT}}\right)^{'}(\cdot)\doteq\lambda_{t_{0}}\pm\epsilon\right\} \right]\right)\geq0.9P(B).\]
Set $\mathcal{D}=\cup_{l=1}^{\frac{m_{\fT}}{N_{\fT}}}\left[B\cap T^{-lN_{\fT}}B\cap\left\{ \left(T^{lN_{\fT}}\right)^{'}(\cdot)\doteq\lambda_{t_{0}}\pm\epsilon\right\} \right]$
and define a map \\
$\phi:\mathcal{D}\to\bbN$ by,\[
\phi(w)=\min\left\{ k\in\left\{ lN_{\fT}\right\} _{l=2}^{\frac{m_{\fT}}{N_{\fT}}}:\ T^{k}(w)\in B\ and\ \left(T^{k}\right)'(w)\doteq\lambda_{t_{0}}\pm\epsilon\right\} .\]
Define an automorphism $S:\mathcal{D}\to B$ by \[
S(w)=T^{\phi(w)}(w).\]
Write $\mathfrak{J}$ for the image $S(\mathcal{D}\cap A)$. We have
$\mathcal{D},\mathcal{\mathfrak{J}}\subset B$ , and \[
P(A\cap\mathcal{D})\geq0.8P(B).\]
Moreover $0.9<\left(S\right)'(w)=\lambda_{t_{0}}\pm\epsilon<1.5,$
and therefore \[
P\left(\mathfrak{J}\mathcal{\cap}\left(B\backslash A\right)\right)\leq0.15P\left(B\setminus A)\right)\leq0.15P(B).\]
 Also \[
P\left(\mathfrak{J}\right)\geq0.9P(\mathcal{D}\cap A)\geq0.72P(B).\]
Therefore \[
P\left(\mathfrak{J}\cap A\right)=P\left(\mathfrak{J}\right)-P\left(\mathfrak{J}\cap\left(B\backslash A\right)\right)\geq0.57P(B).\]
This together with $\left(S\right)'(w)<1.5$ means that \[
P\left(S^{-1}\left(\mathfrak{J}\cap A\right)\right)\geq\frac{0.57}{1.5}P(B)\geq0.35P(B).\]
 So \[
\sum_{l=2}^{\frac{m_{\fT}}{N_{\fT}}}P\left(A\cap\left\{ w:\phi(w)=lN_{\fT}\right\} \cap T^{-lN_{\fT}}A\cap\left\{ w:\left(T^{lN_{\fT}}\right)'(w)\doteq\lambda_{t_{0}}\pm\epsilon\right\} \right)\]
\[
\geq P\left(\left(A\cap\mathcal{D}\right)\cap S^{-1}\left(A\cap\mathfrak{J}\right)\right)\geq0.15\cdot P(B).\]
Whence there is an $2\leq l\leq\frac{m_{\fT}}{N_{\fT}}$, so that\[
P\left(A\cap T^{-lN_{\fT}}A\cap\left\{ w:\left(T^{-lN_{\fT}}\right)'(w)=\lambda_{t_{0}}\pm\epsilon\right\} \right)>0.\]
Hence $\lambda_{t_{0}}$ belongs to the ratio set of $T$ . \end{proof}

\end{document}